\documentclass[12pt,reqno]{amsart}
\usepackage{amsfonts,amssymb, verbatim,amsmath,amsthm,latexsym,textcomp,amscd,hyperref}
\usepackage{latexsym,amssymb,epsfig,verbatim}
\usepackage{amsmath,amsthm,amssymb,latexsym,graphics,textcomp}
\usepackage{blindtext}
\usepackage[left=3.1cm,right=3.1cm,top=3cm,bottom=3cm]{geometry}
\usepackage{graphicx}
\usepackage{color}
\usepackage{url}
\usepackage{pgfplots}
\pgfplotsset{compat=1.18}
\usepackage{tikz}
\usepackage{tikz-cd}
\usepackage[utf8]{inputenc}
\usepackage{mathscinet}
\usepackage{srcltx}

\usepackage{mathabx}
\begin{document}
	\newtheorem{theorem}{Theorem}[section]
	\newtheorem{prop}[theorem]{Proposition}
	\newtheorem{lemma}[theorem]{Lemma}
	\newtheorem{cor}[theorem]{Corollary}
	\newtheorem{qn}[theorem]{Question}
	\theoremstyle{remark}
	\theoremstyle{definition}
	
	\newtheorem{prob}[theorem]{Problem}
	\newtheorem{defn}[theorem]{Definition}
	\newtheorem{notation}[theorem]{Notation}
	\newtheorem{fact}[theorem]{Fact}
	\newtheorem{conj}[theorem]{Conjecture}
	\newtheorem{claim}[theorem]{Claim}
	\newtheorem{example}[theorem]{Example}
	\newtheorem{rem}[theorem]{Remark}
	\newtheorem{assumption}[theorem]{Assumption}
	\newtheorem{scholium}[theorem]{Scholium}
	
	\newtheorem{conv}[theorem]{Notation and Convention}
	
	\newcommand{\HAT}{\widehat}
	\newcommand{\map}{\rightarrow}
	\newcommand{\C}{\mathcal C}
	\newcommand\AAA{{\mathcal A}}
	\newcommand{\SSS}{\mathfrak{S}}
	\newcommand{\bdry}{\partial(G,\mathfrak{S})}
\newcommand{\XX}{\mathcal{X}}
\newcommand{\CC}{\mathcal{C}}
\newcommand{\PP}{\mathcal{P}}
\newcommand{\BB}{\mathcal{B}}
\newcommand{\YY}{\mathcal{Y}}
\newcommand{\mybar}[1]{\smash{\rlap{$\overline{\phantom{#1}}$}}#1}
	\title[On connectedness of the boundary of HHS]{On the connectedness of the boundary of hierarchically hyperbolic spaces}

	\author{Ravi Tomar}
	\email{ravitomar547@gmail.com}
	\address{Beijing International Center for Mathematical Research, Peking University, No. 5 Yiheyuan Road Haidian District, Beijing, P.R.China 100871}
	


	
	
	\begin{abstract} 
	We prove that, under a mild assumption, any metrizable compactification of a one-ended proper geodesic metric space is connected. As a consequence, we deduce that the boundary, introduced by Durham--Hagen--Sisto, of a one-ended hierarchically hyperbolic space is connected. Moreover, we prove that the connectedness of the boundary of a hierarchically hyperbolic group is equivalent to the one-endedness of the group. 
	As an application, we show that if, for $n\geq 2$, $G_1=A_1\ast\dots\ast A_n$ and $G_2=B_1\ast\dots\ast B_n$ are free products of one-ended hierarchically hyperbolic groups, then the boundary of $G_1$ is homeomorphic to the boundary of $G_2$ if and only if the boundary of $A_i$ is homeomorphic to the boundary of $B_i$ for $1\leq i\leq n$.
	\end{abstract}
	\maketitle
	\section{Introduction}
	Motivated by the seminal work of Masur--Minsky \cite{masur-minsky-hyp-cc,masur-minsky-ccii}, Behrstock--Hagen--Sisto introduced the notion of a {\em hierarchically hyperbolic} space and group \cite{BHSI}. This provides a common framework to study mapping class groups and cubical groups.
	In \cite{DHS-boundary}, Durham--Hagen--Sisto introduced a boundary of a hierarchically hyperbolic group that coincides with the Gromov boundary when the group is hyperbolic. This boundary also gives a compactification of the group, and is called the hierarchical boundary of a hierarchically hyperbolic group. In \cite{ABR-thicknesshierarchically}, Abbott--Behrstock--Russell showed that if a hierarchically hyperbolic group $G$ is hyperbolic relative to a collection of subgroups, then the Bowditch boundary of $G$ is a quotient of the hierarchical boundary of $G$. This fact is also crucial to prove our main result. 
	
	It is well known that the Gromov boundary of a hyperbolic group is connected if and only if it is one-ended. The same is true for the Bowditch boundary of a relatively hyperbolic group \cite[Theorem 10.1]{bowditch-relhyp}. So it is natural to look for the relationship between the hierarchical boundary of a hierarchically hyperbolic group and its ends. In \cite{ABD-largest-acyl-actions}, Abbott--Behrstock--Durham introduced a `maximized hierarchical structure' for a given hierarchically hyperbolic space. Using this maximalization, Abbott--Behrstock--Russell \cite[Corollary 5.6]{ABR-thicknesshierarchically} proved that if the hyperbolic space associated to the maximal nested element in the maximized hierarchical structure of a hierarchically hyperbolic group $G$ is one-ended, then the hierarchical boundary with respect to any hierarchical structure on $G$ is connected. One of our main aims in this note is to directly prove that the hierarchical boundary of a one-ended hierarchically hyperbolic space is connected. However, from the proof and under a mild assumption, we see that this holds for any compactification of a proper geodesic metric space. Let $X$ be a proper geodesic metric space and $\partial X$ be a set such that $\overline{X}:=X\cup\partial X$ is a compactification of $X$, i.e. the inclusion $X\to\overline{X}$ is a topological embedding such that its image is a dense and open subset of $\overline{X}$. We further assume that the topology on $\overline{X}$ is metrizable.
    
    \begin{defn}\label{definition-weakly-visible}
    We call $\partial X$ a {\em weakly visible boundary of $X$} if the following holds:

    Suppose $\{x_n\}$ and $\{y_n\}$ are uniformly bounded sequences in $X$. If $\{x_n\}$ is converging to $\xi\in\partial X$ then $\{y_n\}$ is converging to $\xi.$
    \end{defn}
    
The Gromov boundary of a proper hyperbolic space, the Bowditch boundary of a proper relatively hyperbolic space, and the hierarchical boundary of a proper hierarchically hyperbolic space are weakly visible. In this note, the following is our first main result.
\vspace{.2cm}

{\bf Theorem 1} (Theorem \ref{main-theorem-1}){\em A weakly visible boundary of a one-ended proper geodesic metric space is connected.
}
\vspace{.2cm}

In \cite{hamentadt-z-structure}, Hamenst\"{a}dt introduced a $Z$-boundary of the mapping class group of a surface of finite type. In Lemma \ref{lemma-z-boundary}, we show that this $Z$-boundary is weakly visible, and hence by Theorem 1, it is connected (Proposition \ref{proposition-z-boundary-connected}). As an application of Theorem 1 and \cite[Theorem 1.3]{ABR-structure-invariant}, we obtain the following, and answer a question that appears in \cite{ABR-problem-list}:
	\vspace{.2cm}

{\bf Theorem 2} (Theorem \ref{main-theorem-2}). {\em Let $G$ be a hierarchically hyperbolic group. Then, the hierarchical boundary of $G$ is connected if and only if $G$ is one-ended.}
		
\vspace{.2cm}
		
For the definition of a hierarchically hyperbolic group and its boundary, one is referred to Section \ref{2}. In \cite[p. 3672]{DHS-boundary}, the authors conjectured Theorem 2. Here, we prove their conjecture. An application of Theorem 2 implies that the hierarchical boundary of the mapping class group of a connected orientable surface of finite type is connected, see Corollary \ref{corollary-mcg-connected} for the precise statement.
	\vspace{.2cm}
	
In \cite{martin-swiat}, the authors proved that the topology of the Gromov boundary of a free product of hyperbolic groups is uniquely determined by the topology of the Gromov boundaries of the free factors. Zbinden \cite{zbiden-morse}, Tomar \cite{tomar-homeo-rel}, and Chakraborty--Tomar \cite{subhojit-tomar-homeo-floyd} proved similar results for the Morse boundary, Bowditch boundary, and Floyd boundary of a free product of groups, respectively. We also prove a result of the same flavour for hierarchical boundary (Theorem \ref{homeo-free-product}). As an application of this result, we prove the following:
	
	\vspace{.2cm}
	
	{\bf Theorem 3} (Theorem \ref{main-application-1}.) {\em For $n\geq 2$, suppose $G_1=A_1\ast\dots\ast A_n$ and $G_2=B_1\ast\dots\ast B_n$, where $A_i$ and $B_i$ are one-ended hierarchically hyperbolic groups for all $i$. Suppose $G_1$ and $G_2$ have hierarchical structures as described in Subsection \ref{4.2}. Then, the hierarchical boundary of $G_1$ is homeomorphic to the hierarchical boundary of $G_2$ if and only if the hierarchical boundary of $A_i$ is homeomorphic to the hierarchical boundary of $B_i$ for all $1\leq i\leq n$.}
	
	\begin{rem}
		It is possible to give more than one hierarchical structure on a given group. However, it remains an open question whether different hierarchical structures on the same group yield homeomorphic hierarchical boundaries \cite[Question 1]{DHS-boundary}. In this paper, for one-ended groups, we give equivalent conditions for this question using free products (see Corollary \ref{main-cor}). Throughout the paper, for free products of hierarchically hyperbolic groups, we take the hierarchical structure as described in Subsection \ref{4.2}. 
	\end{rem}

\section{Background}\label{2}
In this section, we collect the necessary definitions and results. The definition of {\em hierarchically hyperbolic space (HHS)} is rather technical, and we refer the reader to \cite[Definition 1.2]{BHSII} for a complete account. Roughly, an HHS is an $E$-quasigeodesic metric space with an index set $\SSS$, with some extra data. We include some axioms for being an HHS that are relevant to us. Let $E>0$ and $\XX$ be a $E$-quasigeodesic metric space. Let $\{\CC W:W\in \SSS\}$ be a collection of $E$-hyperbolic spaces.

{\bf Projections.} For each $U\in\SSS$, there exists an $E$-coarsely Lipschitz $E$-coarse map $\pi_U:\XX\to \CC U$ such that $\pi_U(\XX)$ is $E$-quasiconvex in $\CC U$.

{\bf Nesting.} If $\SSS\neq\phi$, then $\SSS$ is equipped with a partial order $\sqsubseteq$ and it has a unique $\sqsubseteq$-maximal element. If $U,V\in \SSS$ and $U\sqsubseteq V$, then we say that $U$ is {\em nested} in $V$. Moreover, for all $U,V\in\SSS$ with $U \sqsubsetneq V$ there is a specified subset $\rho_{V}^U\subset \CC V$ such that Diam$_{\CC V}(\rho_{V}^U)\leq E$. Also, there is a {\em projection} $\rho_{U}^V:\CC V\to \CC U$

{\bf Orthogonality.} $\SSS$ has a symmetric and antireflexive relation called {\em orthogonality}. We write $U\perp V$ when $U$ and $V$ are orthogonal.

{\bf Transversality.} If $U,V\in\SSS$ are not orthogonal and neither is nested in the other, then we say $U$ and $V$ are {\em transverse}, denoted $U\pitchfork V$. Moreover, for all $U,V\in\SSS$ with $U\pitchfork V$, there are non-empty sets $\rho_{U}^V\subseteq \CC U$ and $\rho_{V}^U\subseteq \CC V$ each of diameter at most $E$. 

{\bf Bounded geodesic image.} For all $U,V\in\SSS$ such that $U\sqsubsetneq V$, and all geodesic $\alpha$ in $\CC V$, either Diam$_{\CC U}(\rho_{U}^V(\alpha))\leq E$ or $N_E(\rho_{V}^U)\cap\alpha\neq\emptyset$.

We use $\SSS$ to denote the hierarchically hyperbolic space structure, including the index set $\SSS$,
spaces $\{\CC W: W\in \SSS\}$, projections $\{\pi_W:W\in \SSS\}$, and relations $\sqsubseteq, \perp, \pitchfork$. A quasigeodesic metric space $\XX$ is said to be {\em hierarchically hyperbolic space with constant $E$} if there exists a
hierarchically hyperbolic structure on $\XX$ with constant $E$. The pair $(\XX,\SSS)$ denotes a hierarchically
hyperbolic space equipped with the specific HHS structure $\SSS$.

A {\em hierarchically hyperbolic group (HHG)} is a finitely generated group $G$ that acts on an HHS $(\XX,\mathfrak{S})$ such that
\begin{enumerate}
	\item the action of $G$ on $\XX$ is geometric,
	\item  $G$ acts on $\SSS$ by a $\sqsubseteq$-,$\perp$-, and $\pitchfork$-preserving bijection, and $\mathfrak{S}$ has finitely many $G$-orbits,
	\item the action is compatible with the HHS structure on $\XX$ \cite[p. 483]{petyt-spriano}.
\end{enumerate} 
 
For a precise definition, see \cite[Definition 1.21]{BHSII}. In this case, we say that $\SSS$ is an HHG structure for the group $G$ and use the pair $(G,\mathfrak{S})$ to denote the hierarchically hyperbolic group $G$ equipped with the specific hierarchically hyperbolic group structure $\SSS$. From condition (1) in the definition of an HHG, it follows that $G$ is quasiisometric to $X$. From here, one can give a hierarchical structure on the Cayley graph of $G$, and the left action of $G$ on the Cayley graph also satisfies all the conditions for being an HHG. Thus, to define an HHG, one can use a Cayley graph of $G$ itself. Also, it is easy to show that the definition of HHG does not depend on the choice of the Cayley graph. 

{\bf Hierarchical boundary.} In \cite{DHS-boundary}, the authors introduced the notion of a boundary of an HHS. From \cite[Section 2]{DHS-boundary}, we recall the definition of the hierarchical boundary and its topology. For $S\in \SSS$, $\partial\CC S$ denote the Gromov boundary \cite{gromov-hypgps} of $\CC S$.
\begin{defn}
	Let $(\XX,\SSS)$ be an HHS. A subset $\overline{S}\subset\SSS$ is said to be a {\em support set} if $S_i\perp S_j$ for all $S_i,S_j\in \overline{S}$. Given a support set $\overline{S}$, the {\em boundary point with support $\overline{S}$} is a formal sum $p=\sum_{S\in\overline{S}} a_S^pp_S$, where $p_S\in\partial\CC S$, $a_S>0$, and $\sum_{S\in\overline{S}}a_S^p=1$. By \cite[Lemma 2.1]{BHSII}, such sums are necessarily finite. We denote the support of the boundary point $p$ by Supp($p$). The {\em hierarchical boundary} $\partial(\XX,\SSS)$ of $(\XX,\SSS)$ is the set of all boundary points.
\end{defn}
When the specific HHS structure is clear, we write $\partial\XX$ instead of $\partial(\XX,\SSS)$. For an HHG $(G,\SSS)$, let $(\XX,\SSS)$ be a corresponding HHS. Then, the hierarchical boundary $\partial(G,\SSS)$ of $(G,\SSS)$ is defined to the hierarchical boundary of $(\XX,\SSS)$.

{\bf Topology on $\partial\XX$.} Before defining the topology, we need the notion of a remote point and boundary projection.

\begin{defn}
	[Remote point] A point $q\in\partial\XX$ is called a {\em remote point} with respect to a support set $\overline{S}$ if Supp($q$)$\cap\overline{S}=\emptyset$ and, for all $S\in\overline{S}$, there exists $T_S\in$ Supp($q$) such that $S\notperp T_S$. The set of all remote points of $\partial \XX$ with respect to $\overline{S}$ is denoted by $\partial^{rem}_{\overline{S}}(\XX)$.
\end{defn}
For a support set $\overline{S}$, we denote $\overline{S}$$^{\perp}$ the set of all $U\in\SSS$ such that $U\perp V$ for all $V\in\overline{S}$. Given a support set $\overline{S}$ and $q\in\partial^{rem}_{\overline{S}}\XX$, let $\overline{S}_q$ denote the union of $\overline{S}$ and the set of all $U\in\overline{S}$$^\perp$ such that $U$ is not orthogonal to some $T_U\in$ Supp($q$).
\begin{defn}
	[Boundary projection]  Define a {\em boundary projection} $\partial\pi_{\overline{S}}(q)\in \prod_{S\in \overline{S}_q}\partial\CC S$
	as follows. Let $q=\sum_{T\in\text{Supp}(q)}a_T^q q_T$. For each $S\in \overline{S}_q$ , let $T_S\in\text{Supp}(q)$ be chosen so that $S$ and $T_S$ are not orthogonal. Define the
$S$-coordinate $(\partial\pi_{\overline{S}}(q))_S$ of $\partial\pi_{\overline{S}}(q)$ as follows:
\begin{enumerate}
	\item  If $T_S\sqsubseteq S$ or $T_S\pitchfork S$, then $(\partial\pi_{\overline{S}}(q))_S=\rho_{S}^{T_S}$.
	\item  Otherwise, $S\sqsubseteq T_S$. Choose a $(1,20E)$–quasigeodesic ray $\gamma$ in $\CC T_S$ joining $\rho_{T_S}^S$ to $q_{T_S}$. By the bounded geodesic image axiom, there exists $x\in\gamma$
	such that $\rho_{S}^{T_S}$
	is coarsely constant on the subray of $\gamma$ beginning at $x$. Let
	$(\partial\pi_{\overline{S}}(q))_S=\rho_{S}^{T_S}(x)$.
\end{enumerate}
\end{defn}
The map $\partial\pi_{\overline{S}}$ is coarsely independent of the choice of $\{T_S\}_{S\in\overline{S}}$ (see \cite[Lemma 2.1]{DHS-boundary}). Now, we are ready to define the topology on $\partial\XX$.

Fix a base point $x_0\in\XX$. We define a neighborhood basis for each point $p=\sum_{S\in\overline{S}}a_S^pp_S$, where $p_S\in\partial\CC S$ for each $S\in\text{Supp}(p)=\overline{S}$. For each $S\in\SSS$, choose a neighborhood $U_S$ of $p_S$ in $\CC S\cup\partial\CC S$, and choose $\epsilon>0$. Now, we define the following three subsets of $\partial \XX$ which contribute in the definition of a neighborhood around $p$.
\begin{defn}
	[Remote part] The {\em remote part} $\mathcal{B}^{rem}_{\{U_S\},\epsilon}(p)$ is the set of all $q\in\partial \XX$ such that the following hold:
	\begin{enumerate}
		\item For all $S\in\overline{S}$, $(\partial\pi_{\overline{S}}(q))_S\in U_S$.
		\item $\sum_{T\in\overline{S}^{\perp}}a_T^q<\epsilon$.
		\item For all $S\in \overline{S}_q$ and $S'\in\overline{S}$, $\bigg|\dfrac{d_S(x_0,(\partial\pi_{\overline{S}}(q))_S)}{d_{S'}(x_0,(\partial\pi_{\overline{S}}(q))_{S'})}-\dfrac{a_S^p}{a_{S'}^p}\bigg|<\epsilon$.
	\end{enumerate} 
\end{defn}
\begin{defn}
	[Non-remote part] The {\em non-remote part} $\mathcal{B}^{non}_{\{U_S\},\epsilon}(p)$ is the set of points $q=\sum_{T\in\text{Supp}(q)}a_T^q q_T\in \partial \XX-\partial^{rem}_{\overline{S}}\XX$ such that the following hold, where $A=\overline{S}\cap\text{Supp}(q)$:
	\begin{enumerate}
		\item For all $T\in A$, $q_T\in U_T$.
		\item For all $T\in A$, $|a_T^p-a_T^q|<\epsilon.$
		\item $\sum_{V\in\text{Supp}(q)- A}a_V^q<\epsilon$.
	\end{enumerate}
\end{defn}
\begin{defn}
	[Interior part] The {\em interior part} $\mathcal{B}^{int}_{\{U_S\},\epsilon}(p)$ is the points $x\in\XX$ such that the following conditions are satisfied:
	\begin{enumerate}
		\item For all $S\in\overline{S}$, $\pi_S(x)\in U_S$.
		\item For all $S,S'\in\overline{S}$, $\bigg|\dfrac{d_S(x_0,\pi_S(x))}{d_{S'}(x_0,\pi_{S'}(x)}-\dfrac{a_S^p}{a_{S'}^p}\bigg|<\epsilon.$
		\item For all $S\in\overline{S}$ and $T\in \overline{S}^{\perp}$, $\dfrac{d_T(x_0,x)}{d_S(x_0,x)}<\epsilon$.
	\end{enumerate}
\end{defn}
\begin{defn}
	[Topology on $\XX\cup\partial\XX$] For each $p\in\partial\XX$ with support $\overline{S}$, and $\{U_S\}_{S\in\overline{S}}$, $\epsilon>0$ as above, let $$B_{\{U_S\},\epsilon}(p):=\mathcal{B}^{rem}_{\{U_S\},\epsilon}(p)\cup \mathcal{B}^{non}_{\{U_S\},\epsilon}(p)\cup \mathcal{B}^{int}_{\{U_S\},\epsilon}(p).$$ We declare the set of all such $B_{\{U_S\},\epsilon}(p)$ to form a neighborhood basis around $p$. Also, we include the open subsets of $\XX$ in the topology of $\XX\cup\partial\XX$.
\end{defn}
In \cite[Remark 1.3]{hagen-metrizable}, Hagen clarified why this indeed forms a valid neighborhood basis. This topology does not depend on the choice of the base point $x_0$. 

\begin{theorem}\label{cpt-metrizablity-hhs-bdry}
	\textup{(\cite[Theorem 3.4]{DHS-boundary},\cite{hagen-metrizable})} If $\XX$ is proper, then $\overline{\XX}:=\XX\cup\partial X$ is a compact metrizable space. Moreover, $\XX$ is dense in $\overline{\XX}$.
\end{theorem}

Let $(G,\SSS)$ be an HHG, and let $\partial(G,\SSS)$ be its hierarchical boundary. Define $$\overline{G}:=\Gamma\cup\bdry$$ where $\Gamma$ denotes a Cayley graph of $G$ with respect to a finite generating set. By the previous theorem, $\overline{G}$ is a compact metrizable space. We denote this metric by $d_{\triangle}$. Define a map $\pi:G\to \partial(G,\SSS)$ as 
\begin{equation}\label{equation}
	\pi(g) = \xi, \text{ where } \xi \in \partial(G, \SSS) \text{ such that } d_{\triangle}(g, \xi) = d_{\triangle}(g, \partial(G, \SSS))
\end{equation}

Note that for (1), such a $\xi$ exists as $\partial(G,\SSS)$ is compact, but (2) $\xi$ may not be unique; however, uniqueness is not required for our purposes.
\begin{lemma}\label{dense-lemma}
	
	$\pi(G)$ is dense in $\partial(G,\SSS)$.
\end{lemma}
\begin{proof}
 Let $\xi\in\partial G$. Since $G$ is dense in $\overline{G}$, there exists a sequence $\{g_i\}\subset G$ such that $\lim_{i\to\infty}d_{\triangle}(g_i,\xi)=0$. Thus, by (1), $\lim_{i\to\infty}d_{\triangle}(\pi(g_i),\xi)=0$. Hence, $\pi(G)$ is dense in $\partial (G,\SSS)$.
\end{proof}
We finish this section by recording the following fact that is relevant to us. For relative hyperbolicity and Bowditch boundary, one is referred to \cite{bowditch-relhyp}.

\begin{theorem}\label{bowditch-boudary-hhs}
	\textup{(\cite[Theorem 1.3]{ABR-thicknesshierarchically})} Let $G$ be an HHG that is hyperbolic relative to a finite collection of subgroups $\mathcal{P}$. Then, the Bowditch boundary of $G$ with respect to $\mathcal{P}$ is the quotient of hierarchical boundary of $G$ obtained by collapsing the limit set of each coset of a parabolic subgroup to a point.
\end{theorem}

\section{Proof of Theorem \ref{main-theorem-1}}\label{3}
 This section is devoted to proving our main result. 
 
 \begin{theorem}\label{main-theorem-1}
 A weakly visible boundary of a one-ended proper geodesic metric space is connected.
 \end{theorem}
 \begin{proof}
     Let $X$ be a proper geodesic metric space with weakly visible boundary $\partial X$, and $\overline{X}:=X\cup\partial X$ be a compactification of $X$. Let $d$ and $\overline{d}$ denote the metric on $X$ and $\overline{X}$, respectively. Define a map $\overline{\pi}:X\to\partial X$ as $$\overline{\pi}(x)=\xi, \text{ where } \xi\in\partial X \text{ such that } \overline{d}(x,\xi)=\overline{d}(x,\partial X).$$ Note that such a $\xi$ exists as $\partial X$ is compact, but $\xi$ may not be unique; however, uniqueness is not required for our purposes. We prove the theorem by contradiction. Let if possible $$\partial X=V_1\sqcup V_2$$ where $V_1$ and $V_2$ are non-empty disjoint open subsets of $\partial X$. Let $B_i=\overline{\pi}^{-1}(V_i)$ for $i=1,2$. Denote the closure of $B_i$ in $\overline{X}$ by cl$(B_i)$ (with respect to the topology induced by $\overline{d}$).

{\bf Claim:} $B_1$ and $B_2$ are non-empty, disjoint, and satisfy cl$(B_i)=B_i\cup V_i$ for $i=1,2$.

Let $\xi\in$ cl$(B_i)\setminus\{B_i\}$. Then, there exists a sequence $\{b_n\}\subset B_i$ such that $b_n\to \xi$ as $n\to\infty$. An easy application of triangle inequality shows that $\overline{\pi}(b_n)\to \xi$ as $n\to\infty$. This implies that cl$(B_i)\subseteq B_i \cup V_i$. For the converse, let $\xi \in V_1$. Since $X$ is dense in $\overline{X}$, let $\{x_n\}$ be a sequence in $X$ such that $x_n \to \xi$ (in the original topology of $\overline{X}$). As the topology induced by $\overline{d}$ on $\partial X$ is same as the original topology on $\partial X$, $\{x_n\}\to \xi$ with respect to the metric $\overline{d}$. Suppose there exists a subsequence of $\{x_{n}\}$ contained in $B_2$. Then, $\xi \in$ cl$(B_2)$, which in turn implies that $\xi \in V_2$. This gives a contradiction as $V_1$ and $V_2$ are disjoint. Hence, $\{x_n\}$ is eventually contained in $B_1$ and therefore, $\xi\in$ cl$(B_1)$. Similarly, one can show that $V_2\subseteq$ cl$(B_2)$. This also shows that $B_1$ and $B_2$ are non-empty disjoint subsets of $X$. Hence, the claim.

Since $V_1$ and $V_2$ are non-empty, let $\xi_1 \in V_1$ and $\xi_2 \in V_2$. As in the proof of the claim, there exist sequences $\{x_n\}$ and $\{y_n\}$ in $B_1$ and $B_2$, respectively, such that $x_n \to \xi_1$ and $y_n \to \xi_2$. Since closed $d$-balls in $X$ are compact, up to passing to subsequences, we can assume that $\{x_n\}$ and $\{y_n\}$ are unbounded in $X$. For every $m\in \mathbb{N}$, denote $K_m$ the closed $d$-metric ball of radius $m$ about a fixed base point $x_0\in X$. Since $X$ is one-ended, there exist subsequences $\{x_{n_m}\}$ and $\{y_{n_m}\}$ and a sequence of geodesics $\{\gamma_{m}\}$ joining $x_{n_m}$ and $y_{n_m}$ such that $\gamma_m$ is contained in $X\setminus K_m$. Let $[a_m,b_m]$ be a subpath of $\gamma_{m}$ such that $a_m\in B_1$ and $b_m\in B_2$ and $d(a_m,b_m)\leq 1$ (such a subpath always exists as $B_1\cup B_2$ cover the whole $X$). Note that the sequence $\{a_m\}$ is unbounded in $X$. Since $\overline{X}$ is compact, there exists a subsequence $\{a_{m_k}\}$ that converges to $\eta \in \partial X$. Since $d(a_{m_k},b_{m_k})\leq 1$, it follows $b_{m_k} \to \eta$ in $\overline{X}$ as $\partial X$ is weakly visible that also $\{a_{m_k}\}\to \eta$ in the topology induced by the metric $\overline{d}$ ( since $\overline{d}$ induces the original topology on $\partial X$). Thus, $\eta\in V_1$. By the same logic, $\eta \in V_2$. This gives us a contradiction as $V_1$ and $V_2$ are disjoint. Hence, we have the desired result.   
 \end{proof}
Let $S_g^n$ be the connected orientable surface of genus $g$ and punctures $n$ such that $3g+n-3\geq 2$. Let Mod$(S_g^n)$ and $\mathcal{T}(S_g^n)$ denote the mapping class group and the Teichm\"{u}ller space of $S_g^n$, respectively. For $\epsilon<\epsilon_0$, $\mathcal{T}_{\epsilon}(S_g^n)$ denote a subset of $\mathcal T(S_g^n)$ containing all those elements whose {\em systole} is at least $\epsilon$ (for constant $\epsilon$ and $\epsilon_0$, see \cite[p.1]{hamentadt-z-structure}). By \cite[Theorem 1]{hamentadt-z-structure}, $\mathcal{T}_{\epsilon}(S_g^n)$ is a manifold, and Mod$(S_g^n)$ acts properly and cocompactly on $\mathcal{T}_{\epsilon}(S_g^n)$. Recently, Hamenst\"{a}dt introduced a $Z$-structure for every torsion free finite index subgroup of Mod$(S_g^n)$ \cite[Theorem 4]{hamentadt-z-structure}. In fact, this gives a compactification $\overline{\mathcal{T}(S_g^n)}:=\mathcal{T}_{\epsilon}(S_g^n)\cup X(S_g^n)$ of $\mathcal{T}_{\epsilon}(S_g^n)$. For our purpose, we do not need the full definition of the topology on $\overline{\mathcal{T}(S_g^n)}$. Rather, we just need to know how a sequence of interior points converges to a point of $X(S_g^n)$. This is Definition 4.2 in \cite{hamentadt-z-structure}. Now, we are ready to prove the following:

\begin{lemma}\label{lemma-z-boundary}
    The $Z$-boundary $X(S_g^n)$ is weakly visible.
\end{lemma}
\begin{proof}
    Let $d_{\mathcal{T}}$ denote the Teichm\"{u}ller metric on $\mathcal{T}_{\epsilon}(S_g^n)$. Let $\{X_j\}$ and $\{Y_j\}$ be two sequences in $\mathcal{T}_{\epsilon}(S_g^n)$ such that $d_{\mathcal T}(X_j,Y_j)\leq K$ for some $K\geq 0$. Suppose $X_j\to \xi\in X(S_g^n)$ as $j\to\infty$. We need to show that $Y_j\to \xi$ as $j\to\infty.$ For that, we check conditions $(1)$, $(2)$, and $(3)$ of \cite[Definition 4.2]{hamentadt-z-structure}. The first condition is clear. For conditions (2) and (3), we use Lemma 3.19 and the idea of the proof of Lemma 3.20 of \cite{ABR-structure-invariant}. Using the distance formula \cite[Theorem 6.1]{rafi-combinatorial-model-teich} and the fact that projections to subsurfaces are coarsely Lipschitz, we see that the distance between the projection of $X_j$ and $Y_j$ in the curve graphs of the subsurfaces is uniformly bounded. Then, the lemma follows from \cite[Lemma 3.19]{ABR-structure-invariant}.
\end{proof}

Now, Theorem \ref{main-theorem-1} implies the following:
\begin{prop}\label{proposition-z-boundary-connected}
    The $Z$-boundary $X(S_g^n)$ is connected. \qed
\end{prop}

Since the hierarchical boundary of a proper HHS is weakly visible \cite[Lemma 3.20]{ABR-structure-invariant}, we immediately have the following:
 
 \begin{cor}\label{corollary-hhs-one-ended-connected}
     The hierarchical boundary of a one-ended proper hierarchically hyperbolic space is connected. \qed
 \end{cor}
 It is known that $\mathcal T(S_g^n)$ with respect to either Teichm\"{u}ller metric or Weil--Petersson metric is an HHS \cite[Theorem G]{BHSI}. Since $\mathcal T(S_g^n)$ is one-ended, we have the following:
 \begin{cor}\label{corollary-teich-connected}
     The HHS boundary of $\mathcal T(S_g^n)$ is connected.   \qed
 \end{cor}
 For hierarchically hyperbolic groups, we prove the converse of Corollary \ref{corollary-hhs-one-ended-connected}.
\begin{theorem}\label{main-theorem-2}
	Let $(G,\SSS)$ be an HHS. The hierarchical boundary $\partial (G,\SSS)$ is connected if and only if $G$ is one-ended.
\end{theorem}

\begin{proof}
Suppose $G$ is one-ended. Then, by Corollary \ref{corollary-hhs-one-ended-connected}, $\partial (G,\SSS)$ is connected.
Conversely, suppose $\partial (G,\mathfrak{S})$ is connected. Suppose, if possible, $G$ is not one-ended. Then, there are the following two cases:

{\bf Case 1.} Suppose $G$ is two-ended. Then, $G$ is virtually cyclic. Therefore, $G$ is hyperbolic and by \cite[Theorem 4.3]{DHS-boundary} $\partial (G,\mathfrak{S})$ has only two elements. Thus, this case is not possible.

{\bf Case 2.} Suppose $G$ has infinitely many ends. Then, by \cite[Chapter I, Theorem 8.32(5)]{bridson-haefliger}, $G$ splits as a graph of groups over finite edge groups. If all the vertex groups are finite, then $G$ is virtually a non-Abelian free group, and hence $\partial (G,
\mathfrak{S})$ is homeomorphic to the Cantor set. Since the Cantor set is not a connected space, at least one vertex group has to be infinite. Hence, $G$ is hyperbolic relative to infinite vertex groups \cite{bowditch-relhyp}. Note that the Bowditch boundary of $G$ is disconnected as the coned-off Cayley graph of $G$ with respect to the infinite vertex groups is quasiisometric to the Bass--Serre tree of the splitting of $G$. Now, by Theorem \ref{bowditch-boudary-hhs}, the Bowditch boundary of $G$ is a quotient of $\partial (G,\mathfrak{S})$. Thus, $\partial(G,\mathfrak{S})$ is disconnected. This gives us a contradiction. Hence, this case is also not possible.
Since a finitely generated group, either one-ended, two-ended, or infinite-ended, by the above two cases, we see that $G$ is one-ended.
\end{proof}

It is well known that Mod$(S_g^n)$ is a hierarchically hyperbolic group \cite[Theorem 11.1]{BHSII}. The following is immediate from the above theorem.
\begin{cor}\label{corollary-mcg-connected}
	For $3g+n-3\geq 2$, the HHG boundary of {\em Mod}$(S_g^n)$ is connected.
\end{cor}
\begin{proof}
	Since $3g+n-3\geq 2$, Mod$(S_g^n)$ is neither a hyperbolic nor a relatively hyperbolic group. Hence, Mod$(S_g^n)$ is one-ended. Thus, by Theorem \ref{main-theorem-2}, we are done.
\end{proof}

We end this section with the following remark.
\begin{rem}
	Let $G$ be a finitely generated group that is hyperbolic relative to a finite collection of subgroups $\mathcal{P}$. Let $\partial_{\text{rel}}(G)$ denote the Bowditch boundary of $G$ with respect to $\mathcal{P}$. Suppose $G$ is one-ended. Then, using the same idea of the proof of Theorem \ref{main-theorem-2}, one can show that $\partial_{\text{rel}}(G)$ is connected. This recovers the result of Bowditch, which says that $\partial_{\text{rel}}(G)$ is connected if $G$ does not split non-trivially over finite groups relative to parabolic subgroups.
\end{rem}

\section{Hierarchical boundaries of free products of HHGs}\label{4}
Throughout this section, we fix a free product $G=A\ast B$. When $A$ and $B$ are HHGs, this section aims to give a hierarchical structure of $G$. Using this structure, we then give a description of the hierarchical boundary of $G$. This description is crucial for proving Theorem 3 in the following sections.
\subsection{A model space for a free product of HHGs}\label{4.1}  In this subsection, we associate a graph to the splitting of $G$ that is naturally quasiisometric to a Cayley graph of $G$. This construction can be extended in a straightforward way to free products of finitely many groups.

The definition of the Bass--Serre tree of a graph of groups is classical \cite{serre-trees}. For completeness, we recall it for free products of groups. Let $\tau$ be a unit interval with vertices $v_A$ and $v_B$. We define a tree $T$, called the Bass--Serre tree of the splitting of $G$, as $G\times \tau$ divided by the transitive closure of the following relation $\sim$ $$(g_1,v_A)\sim (g_2,v_A) \text{ if } g_1^{-1}g_2\in A,$$ $$(g_1,v_B)\sim (g_2,v_B) \text{ if } g_1^{-1}g_2\in B.$$

Let $S_A$ and $S_B$ be generating sets of $A$ and $B$, respectively. Let $\Gamma_A$ and $\Gamma_B$ be the Cayley graphs of $A$ and $B$ with respect to $S_A$ and $S_B$, respectively. Define a graph $Y$ as the union of $\Gamma_A,\Gamma_B$ and $\tau$, where $v_A$ is identified with the identity of $A$ and $v_B$ is identified with the identity of $B$. Define a graph $\Gamma$ as $G\times Y$ modulo an equivalence relation induced by 
$$(g_1,y_1)\sim(g_2,y_2) \text{ if }y_1,y_2\in \Gamma_A \text{ and } g_2^{-1}g_1y_1=y_2,$$
$$(g_1,y_1)\sim(g_2,y_2) \text{ if }y_1,y_2\in \Gamma_B \text{ and } g_2^{-1}g_1y_1=y_2.$$
Thus, we obtain a tree of Cayley graphs $\Gamma\to T$, where the preimage of each vertex $v\in T$, called the {\em vertex space} corresponding to $v$, is isometric to the Cayley graph of the stabilizer $G_v$ of $v$ in $G$ (\cite[Subsection 2.1]{martin-swiat}). For $v\in T$, we denote the vertex space corresponding to $v$ by $\Gamma_v$. The following are a few observations about $\Gamma$.

(1) There is a bijection between all the edges of $\Gamma$ connecting different Cayley graphs and the edges of $T$. We call them {\em lifts} of the corresponding edges of $T$. 

(2) By collapsing lifts of all the edges of $T$ to points, we get a natural quotient map from $\Gamma$ to the Cayley graph of $G$ with respect to $S_A\cup S_B$ which is a $G$-equivariant quasiisometry. Also, the natural left action of $G$ on $\Gamma$ is geometric.

(3) Since $G$ is hyperbolic relative to $\{A,B\}$ \cite{bowditch-relhyp}, the graph $\Gamma$ is hyperbolic relative to $\{\Gamma_v:v\in T\}$.

\subsection{HHG structure on $G$}\label{4.2} Let $\SSS_A$ and $\SSS_B$ be HHG structures on $A$ and $B$, respectively. For $g\in G$, let $g\SSS_A$ be
a copy of $\SSS_A$ with its associated hyperbolic spaces and projections in such a way that there is
a hieromorphism (see \cite{BHSII}) $A\to gA$ equivariant with respect to the conjugation isomorphism $A\to A^g$. Similarly, one can put a hierarchical structure on the cosets of $B$ in $G$. For each vertex $v\in T$, we denote the hierarchical structure on $\Gamma_v$ by $\SSS_v$. Since $\Gamma$ is hyperbolic relative to $\{\Gamma_v\}$ and $(\Gamma_v,\SSS_v)$'s are hierarchically hyperbolic, by \cite[Theorem 9.1]{BHSII}, $\Gamma$ is an HHS. We denote this HHS structure on $\Gamma$ by $\SSS$. This implies that $(G,\SSS)$ is an HHG. Now, we briefly explain the HHG structure $\SSS$ on $\Gamma$.

{\bf Indexing set.} Let $\hat{\Gamma}$ denote the graph obtained by coning-off each subspace $\Gamma_v$. Define $$\SSS:=\{\hat{\Gamma}\}\cup(\bigsqcup_{v\in T}\SSS_v).$$

{\bf Hyperbolic spaces.} The hyperbolic space $\CC\hat{\Gamma}$ for $\hat{\Gamma}$ is $\hat{\Gamma}$ itself, while the hyperbolic space for $U\in\SSS_v$, for some $v$, was defined above.

{\bf Relations.} The nesting, orthogonality, transversality relations on each $\SSS_v$ are as above. If $U,V\in\SSS_v,\SSS_w$, and $v\neq w$, then declare $U\pitchfork V$. Finally, for all $U\in \SSS$, let $U\sqsubseteq \hat{\Gamma}$.

{\bf Projections.} For $\hat{\Gamma}$, $\pi_{\hat{\Gamma}}:\Gamma\to\hat{\Gamma}$ is the inclusion which is coarsely surjective. For each $v\in T$, let $\pi_v$ denotes the nearest point projection of $\Gamma$ onto $\Gamma_v$. Then, for $U\in\SSS_v$, define $\pi_U:=\pi_{U,v}\circ\pi_v$, where $\pi_{U,v}:\Gamma_v\to \CC U$ is the projection in $(\Gamma,\SSS_v)$.

{\bf Relative projections.} For $U\in\SSS_v$, $\rho_{\hat{\Gamma}}^U$ is the cone-point corresponding to $\Gamma_v$. If $U,V\in\SSS_v$ then the coarse maps $\rho_{U}^V$ and $\rho_{V}^U$ were already defined. If $U,V\in \SSS_u,\SSS_v$ and $u\neq v$, then $\rho_{V}^U:=\pi_V(\pi_v(\Gamma_u))$ and $\rho_U^V:=\pi_U(\pi_u(\Gamma_v))$. Finally, if, for $U\in\SSS_v$, $U\sqsubsetneq \hat{\Gamma}$ then $\rho_U^{\hat{\Gamma}}:\hat{\Gamma}\to \CC U$ is defined as follows:

(i) If $x\in \Gamma$, then $\rho_U^{\hat{\Gamma}}(x):=\pi_U(x)$.

(ii) If $x$ is the cone-point over $\Gamma_w$ and $v\neq w$. Then, $\rho_{U}^{\hat{\Gamma}}(x):=\rho_{U}^{S_w}$, where $S_w$ is $\sqsubseteq$-maximal element of $\SSS_w$. The cone-point over $\Gamma_v$ may be sent anywhere in $\CC U$.

For $v\in T$, let $\hat{\Gamma_v}$ denote the coned-off graph obtained by coning-off $\Gamma_v$. We already have observed a one-one correspondence between the edges of $T$ and edges in $\Gamma$ connecting different Cayley graphs. We conclude this subsection by noting the following, whose proof is clear, and thus we omit it.
\begin{lemma}\label{coned-off-lemma}
	 Let $\phi:\hat{\Gamma}\to T$ be a map that sends $\hat{\Gamma}_v$ to $v$, and the edges connecting different Cayley graphs to the corresponding edges of $T$. Then, $\phi$ is a continuous quasiisometry. \qed
\end{lemma}
\subsection{Construction of the compactification}\label{4.3}
Suppose $(A,\SSS_A)$ and $(B,\SSS_B)$ are HHGs. Let $\SSS$ be the hierarchical structure as described in Subsection \ref{4.2}. Here, we construct a compact metrizable space which turns out to be the hierarchical boundary of $(G,\SSS)$. For this, we follow the construction of Martin--\'{S}wi\polhk{a}tkowski \cite[Subsection 2.2]{martin-swiat}, the only difference is that we are taking the hierarchical boundary in place of the Gromov boundary. 

{\bf Boundaries of the stabilizers.} Let $\delta_{Stab}(\Gamma)$ be the set $G\times(\partial (A,\SSS_A)\cup\partial (B,\SSS_B))$ divided by the equivalence relation induced by
$$(g_1,\xi_1)\sim (g_2,\xi_2) \text{ if } \xi_1,\xi_2\in \partial A, g_2^{-1}g_1\in A \text{ and } g_2^{-1}g_1\xi_1=\xi_2,$$
$$(g_1,\xi_1)\sim (g_2,\xi_2) \text{ if } \xi_1,\xi_2\in \partial B, g_2^{-1}g_1\in B \text{ and } g_2^{-1}g_1\xi_1=\xi_2.$$
The equivalence class of an element $(g,\xi)$ is denoted by $[g,\xi]$. The set $\delta_{Stab}(\Gamma)$ comes with a natural action of $G$ on the left. This also comes with a natural projection onto the set of vertices of $T$, which sends the boundary of each vertex stabilizer to the vertex. The preimage of each vertex $v\in T$ is denoted by $\partial(\Gamma_v,\SSS_v)$.

Let $\partial T$ denote the Gromov boundary of $T$. Then, we define the {\em boundary} of $\Gamma$ as
$$\delta(\Gamma):=\delta_{Stab}(\Gamma)\sqcup\partial T.$$

Also, we define a set $\overline{\Gamma}$ (which will be called the {\em compactification} of $\Gamma$) as $$\overline{\Gamma}:=\Gamma\cup\delta(\Gamma).$$
This set has a natural action of $G$ and a natural map $\pi_T:\overline{\Gamma}\to T\cup\partial T$, which sends $\Gamma\cup\delta_{Stab}(\Gamma)$ to $T$. The preimage of a vertex $v\in T$ is $\Gamma_v\cup\partial G_v$ that is identified as a set with $\overline{\Gamma_v}.$

{\bf Topology on $\overline{\Gamma}$.} For a point $x\in \Gamma$, we define a basis of open neighborhoods of $x$ in $\Gamma$ to be a basis of open neighborhoods of $x$ in $\overline{\Gamma}$. Now, we define a basis of open neighborhoods for points of $\delta(\Gamma)$. Fix a vertex $v_0$ of $T$.

(1) Let $\xi\in\delta_{Stab}(\Gamma)$. Suppose $v$ is the vertex of $T$ such that $\xi\in\partial(G_v)$. Let $U$ be an open neighborhood of $\xi$ in $\overline{\Gamma_v}$. Define $V_U$ to be the set of all $z\in\overline{\Gamma}$ such that $\pi_T(z)\neq v$ and the first edge of the geodesic in $T$ from $v$ to $\pi_T(z)$ lifts to an edge of $\Gamma$ that is glued to a point of $U$. Then we set $$V_U(\xi):=U\cup V_U.$$ A neighborhood basis of $\xi$ in $\overline{\Gamma}$ is a collection of set $V_U(\xi)$ where $U$ runs over some neighborhood basis of $\xi$ in $\overline{\Gamma_v}$.

(2) Let $\eta\in \partial T$. Let $T_n(\eta)$ be the subtree of $T$ consisting of those elements $x$ of $T$ for which the first $n$ edges of $[v_0,x]$ and $[v_0,\eta)$ are the same. Suppose $u_n(\eta)$ is the vertex on $[v_0,\eta)$ at the distance $n$ from $v_0$. Let $\partial (T_n{(\eta)})$ denote the Gromov boundary of $T_n{(\eta)}$, and let $\overline{T_n(\eta)}=T_n(\eta)\cup\partial(T_n(\eta))$. We define $$V_n(\eta)=\pi_T^{-1}(\overline{T_n(\eta)}\setminus\{u_n(\eta)\}).$$ We take the collection $\{V_n(\eta):n\geq 1\}$ as a basis of open neighborhoods of $\eta$ in $\overline{\Gamma}$.

We skip a verification that the above collections of sets satisfy the axioms for the basis of open neighborhoods, for an idea of proof, one is referred to \cite[Theorem 6.17]{martin1}. We denote this topology by $\tau$ on $\delta(\Gamma)$. A proof of the following lemma is the same as \cite[Lemma 3.3]{martin-swiat}.
\begin{lemma}
	The space $(\delta(\Gamma),\tau)$ is Hausdorff.
\end{lemma}
\subsection{Equivalence of two topologies on $(G,\SSS)$} From the hierarchical structure of $G$, as a set, we see that $$\partial(G,\SSS)=(\bigsqcup_{v\in T}\partial(G_v,\SSS_v))\cup\partial\hat{\Gamma}.$$ Let $\phi$ be the map as in Lemma \ref{coned-off-lemma}. Hence, it induces a homeomorphism from $\partial\hat{\Gamma}\to\partial T$. Thus, we have a continuous map $\bar{\phi}:\hat{\Gamma}\cup\partial\hat{\Gamma}\to T\cup\partial T$. Let $\mathcal T$ be the topology $\overline{G}:=\Gamma\cup\partial(G,\SSS)$ as defined in Section \ref{2}. Define a natural map $\psi:\overline{G}\to\Gamma\cup\delta(\Gamma)$ in the following manner:

	\begin{equation*}
	\psi(x)=
	\begin{cases}
		x & \text{if } x \in \Gamma\cup(\bigsqcup_{v\in T}\partial G_v),\\
		 \bar{\phi}(x)& \text{if }  x\in \partial\hat{\Gamma}
	\end{cases}
\end{equation*}

\begin{prop}
    $\psi$ is a homeomorphism.
\end{prop} 

\begin{proof}
    Clearly, $\psi$ is a bijection.
Since $\overline{G}, \overline{\Gamma}$ are compact Hausdorff spaces, to show that $\psi$ is a homeomorphism, it is sufficient to show that $\psi$ is continuous. It is continuous on the points of $\Gamma$. Thus, we have the following two cases to consider:

{\bf Case 1.} Let $p\in \partial G_v$ for some $v\in T$. Let $p=\sum_{S\in\overline{S}}a_S^pp_S$, where $\overline{S}$ is the support set of $p$ in $\SSS_v$. Let $\epsilon>0$ and $U_S$ be a neighborhood of $p_S$ in $\partial\CC S$ such that $\BB_{\{U_S\},\epsilon}^v(p)$ is a neighborhood of $p$ in $\overline{\Gamma_v}$. For $U:=\BB_{\{U_S\},\epsilon}^v(p)$, let $V_U(p)=U\cup V_U$ be a neighborhood of $p$ in $\overline{\Gamma}$. We show that the neighborhood $\BB_{\{U_S\},\epsilon}(p)$ of $p$ in $\overline{G}$ satisfies $\psi(\BB_{\{U_S\}}(p))=V_U(p)$.

{\bf Remote part.} Note that $\BB_{\{U_S\},\epsilon}^{rem,v}(p)\subset\BB_{\{U_S\},\epsilon}^{rem}(p)$. Every point in $(\bigsqcup_{v\neq w\in T}\partial G_w)\cup\partial\hat{\Gamma}$ is remote with respect to $\overline{S}$. From the definition of neighborhoods and the hierarchical structure of $\Gamma$, it follows that, for $w\neq v$, a point $q\in\partial G_w$ belongs to $\BB_{\{U_s\},\epsilon}^{rem}(p)$ if and only the vertex of the lift of $e$ attached to $\Gamma_v$ belongs to $\BB_{\{U_S\},\epsilon}^{int,v}(p)\subset U$, where $e$ is the first edge of the geodesic segment $[v,w]\subset T$. Hence, such $q\in\BB_{\{U_s\},\epsilon}^{rem}(p)$ if and only $q\in V_U$. Similarly, a point $\xi\in \hat{\Gamma}$ is in $\BB_{\{U_s\},\epsilon}^{rem}(p)$ if and only if the vertex of the lift of $e$ attached to $\Gamma_v$ belongs to $\BB_{\{U_S\},\epsilon}^{int,v}(p)\subset U$, where $e$ is the first edge of the geodesic ray $[v,\phi(\xi))\subset T$. Hence, such $\xi\in \BB_{\{U_S\},\epsilon}^{rem}(p)$ if and only $\phi(\xi)\in V_U$.

{\bf Non-remote part.} Since non-remote points are in $\partial G_v$, it is clear that $\BB_{\{U_S\},\epsilon}^{non}(p)=\BB_{\{U_S\},\epsilon}^{non,v}(p)$.

{\bf Interior part.} The interior part of $\BB_{\{U_S\},\epsilon}^v(p)$ lies in the interior part of $\BB_{\{U_S\},\epsilon}(p)$. Again, from the definition of neighborhoods and the hierarchical structure of $\Gamma$, it follows that, for $w\neq v$, a point $q\in\Gamma_w$ belongs to $\BB_{\{U_s\},\epsilon}^{int}(p)$ if and only the vertex of the lift of $e$ attached to $\Gamma_v$ belongs to $\BB_{\{U_S\},\epsilon}^{int,v}(p)\subset U$, where $e$ is the first edge of the geodesic segment $[v,w]\subset T$. Hence, such $q\in\BB_{\{U_s\},\epsilon}^{int}(p)$ if and only $q\in V_U$. 

From the above discussion, it follows that $\psi(\BB_{\{U_S\},\epsilon}(p))=V_U(p)$ and hence $\psi$ is continuous at $p$.

{\bf Case 2.} Let $\xi\in\partial\hat{\Gamma}$ and let $\bar{\phi}(\xi)=\eta\in\partial T$. Let $V_n(\eta)$ be a neighborhood of $\eta$ in $\overline{\Gamma}$. Choose a neighborhood $U\subset\hat{\Gamma}\cup\partial\hat{\Gamma}$ such that $\bar{\phi}(U)\subseteq T_n(\eta)$. Note that Supp$(\xi)=\{\hat{\Gamma}\}$ and each point in $\bigsqcup_{v\in T}\partial G_v$ is a remote point with respect to $\{\hat{\Gamma}\}$. We use $U$ to construct the required neighborhood of $\xi$ in $(\overline{G},\mathcal{T})$. In each of the following parts, the conditions involving $\epsilon$ are vacuous. Thus, we remove the dependency of the neighborhood on $\epsilon$.

{\bf Remote part.} Note that, for each $v\in T$, each domain in $\SSS_v$ is nested in $\hat{\Gamma}$ and each element in $\partial (G_v,\SSS_v)$ is remote with respect to $\{\hat{\Gamma}\}$. Also, for $V\in\SSS_v$, $\rho_{\hat{\Gamma}}^V$ is the cone-point corresponding to $\Gamma_v$. Thus, the remote part $$\BB_{U}^{rem}(\xi)=\{\xi'\in\bigsqcup_{v\in T}\partial G_v:\text{ if $\xi'\in \partial G_v$ then the cone-point corresponding to $\Gamma_v$ is in } U\}.$$

{\bf Non-remote part.} Elements in $\partial\hat{\Gamma}$ are the only non-remote points in $\partial (G,\SSS)$. Thus, the non-remote part 
$$\BB_{U}^{non}(\xi)=\{\xi'\in\partial\hat{\Gamma}:\xi'\in U\}.$$

{\bf Interior part.} $\BB_U^{int}(\xi)=\{x\in \Gamma: x\in U\}.$

Let $\BB_{U}(\xi)=\BB_{U}^{rem}(\xi)\cup\BB_{U}^{non}(\xi)\cup\BB_{U}^{int}(\xi)$. Then, from the definition of the neighborhood in $(\overline{\Gamma},\tau)$, we see that $\psi(\BB_U(\xi))\subseteq V_n(\eta)$. Thus, $\psi$ is continuous at $\xi$.
\end{proof}
\section{Homeomorphism types of hierarchical boundaries of free products}
Suppose $(G_1,\SSS_1), (G_2,\SSS_2)$ are two HHGs, and $\Gamma_1,\Gamma_2$ are Cayley graphs of $G_1,G_2$, respectively. By Theorem \ref{cpt-metrizablity-hhs-bdry}, $\overline{\Gamma_1}=\Gamma_1\cup\partial(G_1,\SSS_1),\overline{\Gamma_2}=\Gamma_2\cup\partial(G_2,\SSS_2)$ are compact metrizable. Then, one has induced metrics $d_{\triangle_1}$ and $d_{\triangle_2}$ on $G_1\cup\partial(G_1,\SSS_1)$ and $G_2\cup\partial(G_2,\SSS_2)$, respectively. The following lemma is an analogue of \cite[Lemma 4.2]{martin-swiat} in the context of HHGs, which helps to prove Theorem \ref{homeo free product}. In particular, this lemma helps us to define an isomorphism between Bass--Serre trees of free products given in Theorem \ref{homeo free product}.

\begin{lemma}\label{main lemma}
	Let $\partial f:\partial(G_1,\SSS_1)\to\partial (G_2,\SSS_2)$ be a homeomorphism. Then, there is a bijection (need not be a homomorphism) $f:G_1\to G_2$ such that $f(1)=1$ and $f\cup \partial f:G_1\cup\partial(G_1,\SSS_1)\to G_2\cup\partial(G_2,\SSS_2)$ is a homeomorphism.
\end{lemma}
\begin{proof}
	For $i=1,2$, we write $\partial G_i$ in place of $\partial(G_i,\SSS_i)$. Let $\pi:G_1\to\partial G_1$ as defined in Equation (\ref{equation}). Order the elements of $G_1\setminus\{1\}$ and $G_2\setminus\{1\}$ into sequences $\{g_k\}_{k\in\mathbb N}$ and $\{h_k\}_{k\in\mathbb{N}}.$ Define $f(1)=1$. To get the required $f$, iterate the following two steps alternatively.
	
	{\bf Step 1.} Suppose $k$ is the smallest number for which $f(g_k)$ is not yet defined. Since $G_2$ is dense in $G_2\cup\partial G_2$, choose some $l\in\mathbb N$ such that $h_l$ is not an image of any $g_i$ under the map $f$ and $$d_{\triangle_2}(h_l,\partial f(\pi(g_k)))<\dfrac{1}{k}.$$ 
	Then, define $f(g_k)=h_l$.
	
	{\bf Step 2.} Suppose $k$ is the smallest number for which $h_k$ is not chosen as the image of any $g\in G_1$ under $f$. Since, by Lemma \ref{dense-lemma}, $\pi(G_1)$ is dense in $\partial G_1$, $\partial f(\pi(G_1))$ is dense in $\partial (G_2)$. Choose $g\in G_1\setminus\{1\}$ such that $f$ has not yet been defined on $g$ and $$d_{\triangle_2}(h_k,\partial f(\pi(g)))<d_{\triangle_2}(h_k,\partial G_2)+\dfrac{1}{k}.$$ Then, define $f(g)=h_k$.
	
	By performing the above two steps alternatively, we see that $f$ is a bijection. Since $G_1 \cup \partial G_1$ and $G_2 \cup\partial G_2$ are compact and Hausdorff, to prove that $f\cup \partial f$ is a homeomorphism, it is sufficient to prove that $f\cup \partial f$ is continuous.
	
	Clearly, $f\cup\partial f$ is continuous at the points of $G_1$. Now, consider a sequence $\{g_k\}$ in $G_1$ that converges to $\xi\in\partial G_1$ (in the original topology on $\partial G_1$). Since the topology induced by $d_{\triangle_1}$ on $\partial G_1$ is same as the original topology on $\partial G_1$, $\lim_{k\to\infty} d_{\triangle_1}(g_k,\xi)=0$. This, in turn, implies that $\lim_{k\to\infty}d_{\triangle_1}(g_k,\partial G_1)=0$. Thus, 
	$\lim_{k\to\infty}d_{\triangle_1}(g_k,\pi(g_k))=0$ by the definition of $\pi$. Using triangle inequality, we see that $\pi(g_k)$ converges to $\xi$. Thus, by continuity of $\partial f$, $\partial f(\pi(g_k))$ converges to $f(\xi)$. From the definition of $f$, it follows that $\lim_{k\to\infty} d_{\triangle_2}(\partial f(\pi(g_k)),f(g_k))=0$. This implies that $f(g_k)$ converges to $f(\xi)$. Hence, $f\cup \partial f$ is continuous.
\end{proof}

We immediately have the following:
\begin{cor}\label{main cor}
	Let $\partial f:\partial G_1\to\partial G_2$ be a homeomorphism. Then, there exists a bijection $f:G_1\to G_2$ such that the following holds:
	
	Let $\xi\in \partial G_2$ and $U_2$ be an open neighborhood of $\xi$ in $\overline{\Gamma_2}$. Then, there exists a neighborhood $U_1$ of $(\partial f)^{-1}(\xi)$ in $\overline{\Gamma_1}$ such that $f(G\cap U_1)\subset U_2$.
\end{cor}
\begin{proof}
	Let $f$ be the bijection as in Lemma \ref{main lemma} such that $\bar{f}:=f\cup \partial f:G_1\cup\partial G_1\to G_1\cup\partial G_2$ is a homeomorphism. We show that $f$ is the required bijection. Since $U_2$ is an open neighborhood of $\xi\in \partial G_2$, $V:=\bar{f}^{-1}(U_2)$ is an open neighborhood of $(\partial f)^{-1}(\xi)$ in $G_1\cup\partial G_1$. Thus, $(G_1\cup\partial G_1)\setminus V$ is closed in $G_1\cup\partial G_1$ and $G_1\cup\partial G_1$ is closed in $\overline{\Gamma_1}$, $K:=\overline{\Gamma_1}\setminus V$ is compact in $\overline{\Gamma_1}$. Since $(\partial f)^{-1}(\xi)\in \overline{\Gamma_1}\setminus K$, there exists an open neighborhood $U_1$ of $f^{-1}(\xi)$ in $\overline{\Gamma_1}$ such that $U_1$ is disjoint from $K$, and hence $U_1\subset \bar{f}^{-1}(U_2)$. Now, it is clear that $f(G_1\cap U_1)\subset U_2$.
\end{proof}

The remainder of this section is devoted to proving the following theorem.
\begin{theorem}\label{homeo free product}
	Suppose $G_1=A_1\ast B_1$ and $G_2=A_2\ast B_2$ are two free products of HHGs. For $i=1,2$, let $\SSS_{A_i}$ and $\SSS_{B_i}$ be HHG structures on $A_i$ and $B_i$, respectively. Let $\SSS_1$ and $\SSS_2$ be HHG structures on $G_1$ and $G_2$ as described in Subsection \ref{4.2}. If $\partial(A_1,\SSS_{A_1})$ is homeomorphic to $\partial(A_2,\SSS_{B_2})$ and $\partial(B_1,\SSS_{B_1})$ is homeomorphic to $\partial(B_2,\SSS_{B_2})$, then $\partial(G_1,\SSS_1)$ is homeomorphic to $\partial(G_2,\SSS_2)$.
\end{theorem}
{\bf Notation:} Let $\partial f_1:\partial A_1\to\partial A_1$ and $\partial f_2:\partial B_1\to\partial B_2$ be the fixed homeomorphism of the hierarchical boundaries. 
We denote by $q:\partial A_1\cup\partial B_1\to\partial A_2\cup\partial B_2$ the homeomorphism induced by $\partial f_1$ and $\partial f_2$. 
Let $f_1$ and $f_2$ be the bijections provided by Lemma \ref{main lemma}. Let $T_1$ and $T_2$ be the Bass--Serre trees of $A_1\ast B_1$ and $A_2\ast B_2$, respectively. We denote by $\Gamma_1$ and $\Gamma_2$ the trees of Cayley graphs of $G_1$ and $G_2$, respectively as described in Subsection \ref{4.1}. Finally, $\delta(\Gamma_1)$ and $\delta(\Gamma_2)$ denote the spaces as constructed in Subsection \ref{4.3}.

\vspace{.2cm}

{\bf Isomorphism between $T_1$ and $T_2$.} For $i=1,2$, let $\tau_i=[v_i,u_i]$ be the edge of $T_i$ such that $v_i$ and $u_i$ are stabilized by $A_i$ and $B_i$, respectively. Recall that each non-trivial element $g\in G_1$ can be expressed uniquely, in {\em reduced form}, as $g=a_1b_1,\dots,b_n$, with $a_j\in A_1\setminus\{1\},b_j\in B_1\setminus\{1\}$, allowing also that $a_1=1$ and that $b_n=1$. Define a map $\iota:T_1\to T_2$ that maps any edge $a_1b_1,\dots,a_nb_n\tau_1$ to the edge $f_1(a_1)f_2(b_1),\dots,f_1(a_n)f_2(b_n)\tau_2$. It is easy to check that $\iota$ is an isomorphism such that $\iota(\tau_1)=\tau_2$.

Note that if $\xi\in\partial A_1$ and $g\in G_1$, then the set of all representatives of $[g,\xi]\in\delta_{Stab}(\Gamma_1)$ is of the form $(ga,a^{-1}\xi)$ where $a\in A_1$. Similarly, if $\xi\in\partial B_1$ and $g\in G$ then the set of all representatives are of the form $(gb,b^{-1}\xi)$ for $b\in B_1$. When $\xi\in\partial A_1$, we choose a unique $ga=a_1b_1,...,a_nb_n$ for which $n$ is the smallest possible (in this case we have $b_n\neq 1$). When $\xi\in \partial B_1$, we choose a unique $gb=a_1b_1,...,a_nb_n$ for which $b_n=1$. These representatives of an element $[g,\xi]\in\delta_{Stab}(\Gamma_1)$ are called {\em reduced representatives.}
\vspace{.2cm}

{\em Proof of Theorem \ref{homeo free product}.} To prove that $\partial G_1$ is homeomorphic to $\partial G_2$, it is sufficient to prove that $\delta(\Gamma_1)$ is homeomorphic to $\delta(\Gamma_2)$. We define a map $F:\delta(\Gamma_1)\to\delta(\Gamma_2)$ in the following manner:

(1) Let $[g,\xi]\in\delta_{Stab}(\Gamma_1)$ and let $(g,\xi)$ be its reduced representative. Write $g=a_1b_1,\dots,a_nb_n$. Define $$F([a_1b_1,\dots,a_nb_n,\xi])=[f_1(a_1)f_2(b_1),\dots,f_1(a_n)f_2(b_n),q(\xi)].$$

(2) Let $\eta\in \partial T_1$. We can represent it as an infinite word $\eta=a_1b_1,\dots$ such that, for each $n$, the subword consisting of its first $n$ letters corresponds to the $n$-th edge of the geodesic from $v_1$ to $\eta$ via the correspondence $g\to g.\tau_1$. Define $$F((a_1b_1,\dots))=f_1(a_1)f_2(b_1),\dots$$
where the infinite word on the right gives a geodesic ray in $T_2$ starting from $v_2$. 

Observe that the restriction of $F$ to $\partial T_1$ is the same as the map $\partial T_1\to\partial T_2$ induced from the isomorphism $\iota:T_1\to T_2$.

{\bf Claim:} The map $F$ is a homeomorphism.

From the definition of $F$, it follows that $F$ is a bijection. To show that $F$ is a homeomorphism, it is sufficient to prove that $F$ is continuous. There are two cases to be considered:

{\bf Case 1.} Let $\xi\in \delta_{Stab}(\Gamma_1)$ and let $v$ be the vertex of $T_1$ such that $\xi\in \partial G_v$. Let $U_2$ be an open neighborhood of $F(\xi)$ in $\overline{\Gamma_{\iota{(v)}}}$. By Corollary \ref{main cor}, we have an open neighborhood $U_1\subset\overline{\Gamma_{v}}$ of $\xi$. Then, from the definition of neighborhoods in $\overline{\Gamma_1}$ and in $\overline{\Gamma_2}$, it follows that $F(V_{U_1}(\xi)\cap\delta(\Gamma_1))\subset V_{U_2}(F(\xi))\cap\delta(\Gamma_2)$.

{\bf Case 2.} Let $\eta\in\partial T_1$. For an integer $n\geq 1$, consider the subtree $(T_2)_n(F(\eta))\subset T_2$, defined with respect to $v_2$. Let $(T_1)_n(\eta)=\iota^{-1}((T_2)_n(F(\eta)))$, which is a subtree of $T_1$ with respect to the base vertex $v_1$. Again, from the definition of neighborhoods, it follows that $F(V_n(\eta)\cap\delta(\Gamma_1))=V_n(F(\eta))\cap\delta(\Gamma_2)$.

This completes the proof of the claim, and hence the theorem.
\qed
\vspace{.2cm}

A straightforward generalization of Theorem \ref{homeo free product} to the free product of finitely many HHGs gives the following:
\begin{theorem}\label{homeo-free-product}
	For $n\geq 2$, let $G_1=A_1\ast,\dots,\ast A_n$ and $G_2=B_1\ast,\dots,\ast B_n$ be free products of HHGs. For $1\leq i\leq n$, let $\SSS_{A_i}$ and $\SSS_{B_i}$ be HHG structures on $A_i$ and $B_i$, respectively. Suppose $\SSS_1$ and $\SSS_2$ are HHG structures on $G_1$ and $G_2$ as described in Subsection \ref{4.2}. If, for $1\leq i\leq n$, $\partial(A_i,\SSS_{A_i})$ is homeomorphic to $\partial(B_i,\SSS_{B_i})$, then $\partial(G_1,\SSS_1)$ is homeomorphic to $\partial(G_2,\SSS_2)$.
\end{theorem}
\section{Applications}
Suppose $G=A\ast B$ where $(A,\SSS_A)$ and $(B,\SSS_B)$ are HHGs. Then, $(G,\SSS)$ is an HHG with the hierarchical structure $\SSS$ described in Subsection \ref{4.2}. Corresponding to $G$, let $\delta(\Gamma)$ be the space constructed in Subsection \ref{4.3}. Using Theorem \ref{main-theorem-2}, the following proposition describes the connected components of $\delta(\Gamma)$ whose proof is the same as the proof of Proposition 6.3 and Proposition 6.4 from \cite{martin-swiat}. Hence, we skip its proof.

\begin{prop}\label{connected compo} We have the following:
	\begin{enumerate}
		\item  Let $T$ be the Bass--Serre tree of $G$. Then, a point $\eta\in \partial T$ is its own connected component in $\delta(\Gamma)$.
		\item Suppose $A$ and $B$ are one-ended groups. Then, for each vertex $v\in T$, $\partial G_v$ is a connected component of $\delta(\Gamma)$.
	\end{enumerate} 
\end{prop}
Now, we are ready to prove the following:
\begin{theorem}\label{main-application}
	Suppose $G_1$ and $G_2$ satisfy the hypotheses of Theorem \ref{homeo free product}. Additionally, assume that $A_i$ and $B_i$ are one-ended groups for $i=1,2$. Then, $\partial(G_1,\SSS_1)$ is homeomorphic to $\partial(G_2,\SSS_2)$ if and only if $\partial(A_1,\SSS_{A_1})$ is homeomorphic to $\partial(B_1,\SSS_{B_1})$ and $\partial(A_2,\SSS_{A_2})$ is homeomorphic to $\partial(B_2,\SSS_{B_2})$.
\end{theorem}

\begin{proof}
	Let $\delta(\Gamma_1)$ and $\delta(\Gamma_2)$ be the spaces as constructed in Subsection \ref{4.3}. Suppose $\partial(G_1,\SSS_1)$ is homeomorphic to $\partial(G_2,\SSS_2)$. This implies that $\delta(\Gamma_1)$ is homeomorphic to $\delta(\Gamma_2)$. As a homeomorphism maps connected components to connected components, we see that $\partial(A_1,\SSS_{A_1})$ is homeomorphic to $\partial(B_1,\SSS_{B_1})$ and $\partial(A_2,\SSS_{A_2})$ is homeomorphic to $\partial(B_2,\SSS_{B_2})$. The converse is the content of Theorem \ref{homeo free product}.
\end{proof}
Theorem \ref{main-application} can be generalised in a straightforward manner for free products of finitely many HHGs. 

\begin{theorem}\label{main-application-1}
    For $n\geq 2$, suppose $G_1=A_1\ast\dots\ast A_n$ and $G_2=B_1\ast\dots\ast B_n$, where $A_i$ and $B_i$ are one-ended hierarchically hyperbolic groups for all $i$. Suppose $G_1$ and $G_2$ have hierarchical structures as described in Subsection \ref{4.2}. Then, the hierarchical boundary of $G_1$ is homeomorphic to the hierarchical boundary of $G_2$ if and only if the hierarchical boundary of $A_i$ is homeomorphic to the hierarchical boundary of $B_i$ for all $1\leq i\leq n$.
\end{theorem}

By combining Theorem \ref{homeo free product} and Theorem \ref{main-application}, we immediately have the following corollary:

\begin{cor}\label{main-cor}
	Let $G$ be a one-ended group. Suppose $\SSS_1$ and $\SSS_2$ are two hierarchical structures on $G$. Let $\SSS$ and $\SSS'$ be the hierarchical structures on $(G,\SSS_1)\ast (G,\SSS_1)$ and $(G,\SSS_2)\ast(G,\SSS_2)$, respectively, as described in Subsection \ref{4.2}. Then, $\partial(G,\SSS_1)$ is homeomorphic to $\partial(G,\SSS_2)$ if and only if $\partial(G\ast G,\SSS)$ is homeomorphic to $\partial(G\ast G,\SSS')$.
\end{cor}
\subsection{Locally quasiconvex HHGs}
In \cite{BHSII}, the authors introduce the notion of hierarchical quasiconvexity in HHSs.
\begin{defn}
	\textup{(\cite[Definition 5.1]{BHSII})}\label{HQC} Let $(\XX,\SSS)$ be an HHS, and $k:[0,\infty)\to[0,\infty)$ be a map. A subset $\YY$ of $\XX$ is said to be {\em $k$-hierarchically quasiconvex} if the following hold:
	\begin{enumerate}
		\item For each $U\in\SSS$, $\pi_U(\YY)$ is $k(0)$-quasiconvex subset of $\CC U$.
		\item If $x\in\XX$ satisfies $d_U(x,\YY)\leq r$ for each $U\in\SSS$, then $d_{\XX}(x,\YY)\leq k(r)$.
	\end{enumerate}
The subspace $\YY$ is said to be {\em hierarchically quasiconvex (HQC)} if it is $k$-hierarchically quasiconvex for some $k:[0,\infty)\to[0,\infty)$. A subgroup $H$ of an HHG $(G,\SSS)$ is {\em hierarchically quasiconvex} if $H$ is a hierarchically quasiconvex
subset of $G$ equipped with a finitely generated word metric.
\end{defn}
The definition of an HQC subgroup does not depend on the choice of a finite generating set of the ambient group \cite[Proposition 5.7]{RSC-convexity}. Also, an HQC subgroup of an HHG $(G,\SSS)$ is finitely generated and undistorted \cite[Lemma 2.10]{ABR-thicknesshierarchically}.
\begin{defn}
	Let $(G,\SSS)$ be an HHG. We say that $G$ is {\em locally hierarchically quasiconvex} if every finitely generated subgroup of $G$ is hierarchically quasiconvex.
\end{defn}
We suspect that, under natural conditions on the HHG structure, locally HQC HHGs are hyperbolic and locally quasiconvex.  We plan to explore this in future work.

In this subsection, we prove a combination theorem for a free product of locally HQC.
\begin{theorem}\label{theorem-hqc}
	Let $G=A\ast B$, where $(A,\SSS_A)$ and $(B,\SSS_B)$ are locally hierarchically quasiconvex HHGs. Then, with respect to the HHG structure described in Subsection \ref{4.2}, $G$ is locally HQC.
\end{theorem}
\begin{proof}
	Let $\Gamma$ be the graph constructed in Subsection \ref{4.1} and $H$ be a finitely generated subgroup of $G$. Then, $H$ has an induced graph of groups structure whose vertex groups are the intersection of $H$ and the conjugates of $A$ or $B$ in $G$ (some vertex groups may be trivial), and the edge groups are trivial. We continue to use the same notation as in Section \ref{4}. Since $\Gamma_v$'s are isometric to the Cayley graphs of $G_v$'s, we can realize $H$ as a subset of $\Gamma$ in the following manner:
	
	Let $T_H\subset T$ be the Bass--Serre tree of the graph of groups decomposition of $H$. Note that, for each vertex $v\in T_H$, the stabilizer $H_v$ of $v$ in $H$ is $H\cap G_v$. Thus, we can see $H_v$ as a subset of $\Gamma_v$. In this way, we realize $H$ as a subset of $\Gamma$.
	
	 We show that $H$ is an HQC subset of $\Gamma$ which shows that $H$ is an HQC subgroup of $G$. Since $H$ is finitely generated, $H_v$ is finitely generated for each vertex $v\in T_H$. Hence, $H_v$'s are HQC subsets of $(\Gamma_v,\SSS_v)$'s. As there are finitely many vertex groups in the graph of groups decomposition of $H$, we can assume that $H_v$'s are $k$-HQC for some $k:[0,\infty)\to[0,\infty)$. 
	
	{\bf Condition (1) of Definition \ref{HQC}.} From the hierarchical structure on $\Gamma$, it follows that, for $v\in T_H$ and $U\in\SSS_v$, $\pi_U(H)=\pi_U(H_v)$. Since $H_v$ is $k$-HQC, we see that $\pi_U(H)$ is $k(0)$-quasiconvex subset of $\CC U$ for each $v\in T_H$ and all $U\in\SSS_v$. If $v\in T\setminus T_H$, then, for all $U\in\SSS_v$, $\pi_U(H)$ is a subset of diameter $E$ in $\CC U$, where $E$ is a constant in the HHS structure on $\Gamma$. Clearly, $\pi_{\hat{\Gamma}}(H)$ is $C$-quasiconvex in $\hat{\Gamma}$ for some $C\geq 0$. By redefining the function $k$ at $0$ as max$\{k(0),E,C\}$, we see that for all $U\in\SSS$, $\pi_U(H)$ is $k(0)$-quasiconvex in $\CC U$.
	
	{\bf Condition (2) of Definition \ref{HQC}.} Let $x\in \Gamma$ such that $d_U(x,H)\leq r$ for each $U\in \SSS$. If $x\in H$, then there is nothing to prove. We consider the following two cases:
	
	{\bf Case 1.} Suppose $x\notin H$ but $x\in \Gamma_w$ for some vertex $w\in T_H$. By our assumption, $d_U(x,H_w)=d_U(x,H)\leq r$ for all $U\in\SSS_w$. However, $H_w$ is $k$-HQC in $G_w$. This implies that $d_{\Gamma_w}(x,H_w)\leq k(r)$. Since $\Gamma_w$ is isometrically embedded in $\Gamma$, $d_{\Gamma}(x,H_w)=d_{\Gamma}(x,H)\leq k(r)$. 
	
	{\bf Case 2.} Suppose $x\notin \Gamma_w$ for all $w\in T_H$. Suppose $x\in\Gamma_u$ for some vertex $u\in T\setminus T_H$. Let $w\in T_H$ be the closest vertex to $u$ and let $y$ be the point of $\Gamma_w$ to which the first edge of the geodesic $[w,u]\subset T$ is attached. Then, from the HHS structure on $\Gamma$, it follows that $d_U(x,H)=d_U(y,H_w)$ for all $U\in \SSS_w$. Since $H_w$ is $k$-HQC in $G_w$, $d_{\Gamma_w}(y,H_w)\leq k(r)$ and hence $d_{\Gamma}(y,H_w)\leq k(r)$. Observer that, for all $u\neq v\in T$ and $U\in\SSS_v$, $d_U(x,G_w)=0$. For $U\in\SSS_u$, $d_U(x,G_w)=d_U(x,H)\leq r$. Similarly, $d_{\hat{\Gamma}}(x,G_w)=d_{\hat{\Gamma}}(x,H)\leq r$. Since $G_w$ is $k_1$-HQC in $G$ for some $k_1:[0,\infty)\to[0,\infty)$ \cite[Theorem 1.2]{RSC-convexity}, $d_{\Gamma}(x,G_w)\leq k_1(r)$. However, $d_{\Gamma}(x,G_w)=d_{\Gamma}(x,y)$. Thus, using triangle inequality, we see that $d_{\Gamma}(x,H)=d_{\Gamma}(x,H_w)\leq k(r)+k_1(r)$. By redefining the function $k$ at $r$ as $k(r)+k_1(r)$, we are done.
\end{proof}
By induction and using Theorem \ref{theorem-hqc}, we have the following general version of the previous theorem.
\begin{theorem}
	For $n\geq 2$, let $G=A_1\ast\dots\ast A_n$ be a free product of locally hierarchically quasiconvex HHGs. Then, with respect to the HHG structure described in Subsection \ref{4.2}, $G$ is locally HQC.    \qed
\end{theorem}
\vspace{.2cm}
\noindent
{\bf Acknowledgements.} I want to express my deepest gratitude to Mark Hagen for reviewing the draft version of this note and pointing out an error. 
I am also sincerely thankful to both Jason Behrstock and Mark Hagen for their encouragement and support during the writing of this note. Finally, I would like to thank the referee for encouraging me to write down the main result for the compactification of one-ended proper geodesic metric spaces.

\end{document}